\documentclass[12pt,reqno]{amsart}
\usepackage{amssymb}

\usepackage{amscd}

\newcommand{\RNum}[1]{\uppercase\expandafter{\romannumeral #1\relax}}

\usepackage[T2A]{fontenc}
\usepackage[utf8]{inputenc}
\usepackage[russian,english]{babel}
\input{int.def}

\usepackage{tikz-cd}
\usetikzlibrary{cd}
\usepackage{enumitem}
\usepackage{dirtytalk}
\usepackage{hyperref}
\usepackage{xcolor}

\numberwithin{equation}{section}

\def\exd{\text{\normalfont{d}}}

\def\fsw{\sf{FSW}}
\def\sq{\sf{sq}}

\usepackage[mathcal]{euscript}

\usepackage{titlesec}
\titleformat{\section}[runin]{\bfseries}{\thesection.}{3pt}{}[.]

\myitemmargin 
\usepackage{geometry}
\newgeometry{vmargin={25mm}, hmargin={13mm,13mm}}   % set the margins

\begin{document}

\title[From flops to diffeomorphism groups]%
{From flops to diffeomorphism groups}

\author{Gleb Smirnov}
%\address{ETH Zürich}
%\qurraddr{}
%\email{gleb.smirnov@math.ethz.ch}
%\thanks{}

%    \subjclass is required.
%\subjclass[2010]{Primary }

%\date{}

%\dedicatory{}

\begin{abstract}
We exhibit many examples 
of closed complex surfaces whose 
diffeomorphism groups are not simply-connected and contain 
loops that are not homotopic to loops of symplectomorphisms.
\end{abstract}

\maketitle

\setcounter{section}{0}
\section{Introduction}\label{intro}
Let $X_0$ be a closed algebraic surface with a single ordinary 
double-point. Assume that $X_0$ admits a global smoothing 
$f \colon \calx \to \Delta$, $X_t = f^{-1}(t)$, 
where $\Delta \subset \cc$ is a complex disk, such that $f$ has a single isolated singularity, modeled in local 
complex coordinates by
\[
f(x_1,x_2,x_3) = x_1^2 + x_2^2 + x_3^2.
\]
As $X_0$ is a surface, it has a unique minimal resolution $q \colon \tilde{X}_0 \to X_0$. The exceptional locus of this resolution 
is a smooth rational curve $C \subset \tilde{X}_0$ 
of self-intersection number $(-2)$. 
Atiyah \cite{Atia} proved that $\tilde{X}_0$ is diffeomorphic 
to the surface $X_t$ for any $t \in \Delta - \left\{ 0 \right\}$. Suppose that
\begin{equation}\label{eq:my-assumptions}
    \pi_1(\tilde{X}_0) = 0,\quad 
    p_g(\tilde{X}_0) > 1,\quad 
    p_g(\tilde{X}_0) > h^{0}(\scro(K_{\tilde{X}_0} - C)),
\end{equation}
where $p_g(\tilde{X}_0) = h^0(\scro(K_{\tilde{X}_0}))$ is the geometric 
genus of $\tilde{X}_0$; 
the latter inequality of \eqref{eq:my-assumptions} 
means that the canonical bundle 
of $\tilde{X}_0$ has a section that does not vanish identically on $C$. 
Suppose further that $\tilde{X}_0$ is endowed with a symplectic form 
$\omega$, which may or may not be K{\"a}hler. 
As $\symp(X,\omega)$ is a subgroup of $\diff(X)$, we have 
the inclusion induced homomorphisms 
$\pi_k(\symp(X,\omega)) \to \pi_k(\diff(X,\omega))$. This note aims to prove the following
\begin{theorem}\label{thm:main}
Under assumptions \eqref{eq:my-assumptions}, the 
homomorphism $\pi_1 (\symp(\tilde{X}_0,\omega)) \to \pi_1 (\diff(\tilde{X}_0))$ is not surjective.
\end{theorem}
To give an example of a resolution $\tilde{X}_0$ which agrees with 
\eqref{eq:my-assumptions}, we consider a smooth quintic surface in $\cp^3$. 
The geometric genus of a quintic 
can be computed as follows: any canonical divisor of a quintic corresponds 
to a hyperplane section, hence 
$p_g = \dim_{\cc}\,\cp^3 + 1$. Inside the projective space $\cp^{N}$ which parameterizes 
all quintics, there is a codimension-1 locus $\Sigma$ which parameterizes singular quintics. 
Each smooth point of $\Sigma$ corresponds to a quintic with a single double-point singularity. 
Let $\Delta$ be a small complex disk in $\cp^{N}$ that intersects $\Sigma$ transversally at a smooth point $p \in \Sigma$, and 
let $t \colon \Delta \to \cc$ be a local parameter on 
$\Delta$ such that $t(p) = 0$. 
Denote by $X_t$ the quintic corresponding to the point $t \in \Delta$.
Then, the family $\left\{ X_t \right\}_{t \in \Delta}$ gives a global smoothing of $X_0$. 
Let $q \colon \tilde{X}_0 \to X_0$ be the minimal 
resolution of $X_0$, and let $C$ the exceptional $(-2)$-curve. 
Recall that $\pi_1(X_t) = 0$. Both $\pi_1$ and $p_g$ are diffeomorphism invariants. Since $X_t$ is diffeomorphic to $\tilde{X}_0$, it follows that 
$\tilde{X}_0$ satisfies the first two conditions of \eqref{eq:my-assumptions}. 
Let us check that $\tilde{X}_0$ obeys the third condition: Consider a hyperplane $h \subset \cp^3$ which does not pass through the singular point of $X_0$. The intersection 
$H = X_0 \cap h$ is a canonical divisor of $X_0$. 
The minimal resolution being crepant 
($K_{\tilde{X}_0} = q^{*} K_{X_0}$), the divisor 
$q^{*}H$ is canonical. Since $q^{*}H$ is disjoint from $C$, the third inequality of \eqref{eq:my-assumptions} follows.
\smallskip%

To construct a loop in $\diff(\tilde{X}_0)$ that is not represented 
by a loop in $\symp(\tilde{X}_0,\omega)$, we use the Atiyah flop, a birational surgery introduced in \cite{Atia}. 
Consider the ramified double covering of $\calx$:
\begin{equation}\label{eq:branch}
\caln  = \left\{ (t,x) \in \Delta \times \calx\,|\, f(x) = t^2 \right\}
\end{equation}
The 3-fold $\caln$ fibers over $\Delta$ via the map 
$f_{\caln}(t,x) = t$. If $\sq \colon \Delta \to \Delta$ is given by $\sq(t) = t^2$, then $\caln$ corresponds to the base change
\begin{equation}
\begin{tikzcd}\label{d:base-change}
\caln \arrow{r}{} \arrow{d}{f_{\caln}} & \calx \arrow{d}{f}\\
\Delta \arrow{r}{\sq} & \Delta,
\end{tikzcd}
\end{equation}
where the upper horizontal arrow is the covering map given by $(t,x) = x$. 
The 3-fold $\caln$ is a nodal 3-fold in the sense of \cite{Atia}, with a 
single double-point in the fiber over $0$. 
Atiyah shows (see \cite[\S 2]{Atia}) that there exists 
a resolution $r \colon \calv \to \caln$ which replaces the double-point 
by a smooth rational $(-1,-1)$-curve $C$, that is, a curve whose normal bundle 
is $\scro(-1) \oplus \scro(-1)$. Define $p \colon \calv \to \Delta$ by the 
diagram:
\begin{equation}
\begin{tikzcd}[column sep=small]
\calv \arrow{rr}{r} \arrow{dr}[swap]{p} & & \caln \arrow{dl}{f_{\caln}}\\
& \Delta &
\end{tikzcd},
\end{equation}
Atiyah proves 
(see \cite[\S 3]{Atia}) that $p$ has maximal rank everywhere. 
Hence, $\calv \xrightarrow{p} \Delta$ 
is a holomorphic fiber bundle. In particular, 
for each $t \in \Delta$, the fiber $p^{-1}(t)$ is diffeomorphic 
to $p^{-1}(0)$. On the other hand, the restriction 
of $r$ to the fiber $p^{-1}(0)$ gives the minimal 
resolution
\[
r|_{p^{-1}(0)} \colon p^{-1}(0) \to f_{\caln}^{-1}(0) = f^{-1}(0).
\] 
Thus, for the surface  $X_0 = f^{-1}(0)$, one 
can form a smoothing $X_t = f^{-1}(t)$ and the 
minimal resolution $\tilde{X}_0$, and those two are diffeomorphic. 
Detailed proofs of these results are given in \cite{Atia}. 
The resolution $r \colon \calv \to \caln$ is called a small 
resolution of $\caln$ because the double-point of $\caln$ is 
replaced by a curve, a codimension $2$ 
subvariety. It is well known that this process can 
be done in two different 
ways, so we obtain two different 
families of smooth surfaces, $\calv$ and 
$\calv'$, which coincide outside their central fibers.
More precisely, we have:
\begin{theorem}[Burns-Rapoport, \cite{B-R}]\label{thm:burns}
Let $p \colon \calv \to \Delta$ be a 
holomorphic family of smooth complex surfaces. 
For each $t \in \Delta$, set $X_t = p^{-1}(t)$. 
If $X_0$ contains a smooth 
rational $(-2)$-curve $C$ which is embedded in $\calv$ as a $(-1,-1)$-curve, then 
there exists another holomorphic family $p' \colon \calv' \to \Delta$, 
$X^{'}_t = p'^{-1}(t)$, and a birational isomorphism
\[
\rho_{C} \colon \calv \to \calv',
\]
whose indeterminacy locus is $C$, that fits into a diagram
\begin{equation}
\begin{tikzcd}[column sep=small]
\calv \arrow[dashed]{rr}{\rho_C} \arrow{dr}[swap]{p} & & \calv' \arrow{dl}{p'}\\
& \Delta &
\end{tikzcd},
\end{equation}
where the dashed arrow indicates that $\rho_C$ is 
not a map but merely a birational map. 
Although one cannot extend $\rho_C$ into the
curve $C$ to make it a 
proper isomorphism, one can restrict it 
to $X_0$ to obtain 
a birational isomorphism 
\[
\rho \colon X_0 \to X'_0,
\]
which then extends to a proper isomorphism between 
$X_0$ and $X'_0$. Under that isomorphism, 
the image $C' = \rho(C)$ is also a smooth rational curve which 
is embedded in $\calv'$ as a $(-1,-1)$-curve. 
\smallskip%

The family 
$p \colon \calv \to \Delta$ is 
differentiably trivial, so it provides an identification
$\alpha \colon H_2(X_0;\zz) \to H_2(X_t;\zz)$,
where $t \neq 0$ is some fixed base-point. Similarly, 
we have another identification 
$\alpha' \colon H_2(X'_0;\zz) \to H_2(X'_t;\zz)$,
corresponding to the family $p' \colon \calv' \to \Delta$.
If we identify $H_2(X_0;\zz)$ and 
$H_2(X'_0;\zz)$ via the diagram
\begin{equation}
\begin{tikzcd}\label{d:homology-ident}
H_2(X_0;\zz) \arrow{d}{\alpha} \rar{}  & H_2(X'_0;\zz) \arrow{d}{\alpha'} \\
H_2(X_t;\zz) \arrow{r}{\rho_{C *}} & H_2(X'_t;\zz)\,,
\end{tikzcd}
\end{equation}
then the formula 
for $\rho_{*} \colon H_2(X_0;\zz) \to H_2(X'_0;\zz)$ is 
as follows:
\begin{equation}\label{eq:reid}
    A \to A + (A,C)C\quad \text{for each $A \in H_2(X_0;\zz)$,}
\end{equation}
where $(A,C)$ stands for the intersection pairing.
\end{theorem}
The birational map $\rho_C$ is called the elementary 
modification of $\left\{ X_t \right\}_{t \in \Delta}$ 
with the center $C$ or the Atiyah flop. It has an explicit 
description in terms of blowups and 
blowdowns (see, eg, Reid's paper \cite[Part II]{Reid}). 
This theorem is due to Burns-Rapoport; in the present form, however, it is 
a mix of Corollary 2 of Morrison's paper \cite{Mor} and 
Theorem (6.3) of \cite{Reid}.
\smallskip%

Now, let us glue the spaces 
$\calv$ and $\calv'$ along their boundaries to form a space
\begin{equation}\label{eq:make-W}
\calw = \calv \cup \calv'/ \sim \quad 
\text{$(t,x) \in \calv \sim (t,x') \in \calv$ iff $\rho_C(x) = x'$ and $t \in \del \Delta$.}
\end{equation}
Two copies of $\Delta$, glued together by the identity map along their boundaries, 
form a 2-sphere $S^2 = \Delta \cup_{\id} \Delta$. The manifold 
$\calw$ is a fiber bundle over $S^2$ under the projection map
\[
p^{\cup} \colon \calw \to S^2,\quad 
p^{\cup}(t,x) = 
\begin{cases}
      p(t,x) & \text{for $(t,x) \in \calv$} \\
      p'(t,x) & \text{for $(t,x) \in \calv'$.}
\end{cases}
\]
Note that the gluing of $\Delta$'s is orientation-reversing with respect to 
their natural complex orientations, so $\calw$ 
will not be 
a holomorphic bundle, but merely a bundle of complex surfaces.
\smallskip%

Using family Seiberg-Witten invariants, we will prove 
$\calw$ cannot be realized as a Hamiltonian bundle. 
The idea is to use a computation from Kronheimer's paper \cite{K} 
to show that this fiber bundle has non-vanishing Seiberg-Witten invariant; then
to argue that if the bundle admitted fiberwise cohomologous 
symplectic forms, then 
the family Seiberg-Witten invariant would have to vanish. 
This argument is in spirit close to the way 
Ruberman (see \cite{R1, R2}) constructed the first examples of self-diffeomorphisms 
of four-manifolds that are isotopic to 
the identity in the topological category but not smoothly so. A generalization 
of Ruberman's result to higher homotopy groups will appear in a joint 
work of Auckly and Ruberman \cite{R3}.
\smallskip%

Watanabe (see \cite{W})
has recently proved a related result by rather different methods.
He shows that $\pi_k ( \diff_{\text{ct}}(\rr^4) )$ 
are infinite groups for all $k \geq 1$. 
On the other hand, 
Gromov (see \cite{Gr}) proved that 
$\symp_{\text{ct}}(\rr^4,\omega_{st})$ is 
contractible. Hence, the homomorphism 
\begin{equation}\label{map:wata}
\pi_1 ( \symp_{\text{ct}}  (\rr^4,\omega_{st}) )  \to \pi_1 ( \diff_{\text{ct}}(\rr^4) )
\end{equation}
is not surjective, and nor are the 
homomorphisms for higher $\pi_k$. In addition to Watanabe's 
example, the only other example the author is aware of 
where this non-surjectivity arises is Example\,10.4.2 of \cite{McD-Sa-2}. 
Let us consider the product 
$(S^2 \times S^2, \omega_{\lambda} = (1+\lambda) \omega_{S^2} \oplus \omega_{S^2})$, 
where $0 \leq \lambda \in \rr$ and $\omega_{S^2}$ is a 
standard area form on $S^2$ with total area of $1$. 
For $\lambda = 0$, Gromov proved (see \cite{Gr}) 
that the symplectomorphism group retracts onto the isometry group 
$\zz_2 \times \SO(3) \times \SO(3)$. In partucular, 
$\pi_1( \symp(S^2 \times S^2, \omega_0) ) = \zz_2 \oplus \zz_2$. 
In \cite{McD-Sa-2}, McDuff and Salamon 
explicitly describe an element $[\psi_t] \in \pi_1(\diff(S^2 \times S^2))$ that 
is of infinite order. Hence, the homomorphism 
\begin{equation}\label{map:mcduff}
\pi_1( \symp(S^2 \times S^2, \omega_{\lambda}) ) \to 
\pi_1( \diff(S^2 \times S^2) )
\end{equation}
is not surjective for $\lambda = 0$. They also claim 
that $\pi_1(\diff(S^2 \times S^2))$ has rank at least two, for it 
contains both $[\psi_t]$ and 
its conjugate by the involution that interchanges the two $S^2$ factors.
For $\lambda > 0$, Abreu and McDuff proved (see \cite{Abreu,Ab-McD}) that
$\pi_1( \symp(S^2 \times S^2, \omega_{\lambda}) ) = \zz_2 \oplus \zz_2 \oplus \zz$. 
Hence, the homomorphism \eqref{map:mcduff} is not surjective for all $\lambda$.
\smallskip%

To the author's knowledge, Theorem \ref{thm:main} 
provides the first examples, away from the setting of rational or ruled surfaces, 
of symplectic manifolds $(X,\omega)$ for which the 
homomorphism $\pi_1( \symp(X, \omega) ) \to \pi_1( \diff(X) )$ is not surjective.

\statebf Acknowledgements.
\,The author is greatly indebted to Sewa Shevchishin for many stimulating conversations. 
The author also thanks Jianfeng Lin for pointing out an inaccuracy in an earlier draft of this work.

\section{Family Seiberg-Witten invariants}\label{sec:family}
In what follows 
we shall pass to Sobolev completions of all of the 
functional spaces encountered, not mentioning the choice 
of those completions explicitly. 
A general reference that in-depth covers the 
involved analysis is Nicolaescu's book \cite{Nic}.
\smallskip%

One starts with a closed oriented {\itshape simply-connected} $4$-manifold $X$ equipped with a Riemannian metric $g$ and a self-dual form $\eta$. After picking a spin$^{\cc}$ structure on $X$, with associated spinor bundles $W^{\pm}$ and determinant line bundle $\call$, one considers the monopole map
\begin{equation}\label{map:monopole}
\mu \colon \Gamma(W^{+}) \times \cala \to \Gamma(W^{-}) \times i \,\Omega_{+}^2(X),\qquad \mu(\varphi,A) = ( \cald^{A}\,\varphi, F^{+}_A - \sigma(\varphi) - i \eta ),  
\end{equation}
where $\varphi \in \Gamma(W^{+})$ is a self-dual spinor field, $A \in \cala$ is a $\UU(1)$-connection on $\call$, and $F^{+}_{A} \in i\,\Omega_{+}^2(X)$ 
stands for the self-dual part of the curvature form of $A$. 
Finally, $\sigma \colon \Gamma(W^{+}) \to i\,\Omega^{2}_{+}(X)$ is the squaring map. We also write $\mu_{(g,\eta)}$ when we 
want to indicate the dependence of the monopole map on the metric and perturbation. The Seiberg-Witten solution space (space of monopoles) is the zero set of the function $\mu$, while 
the solution moduli space $\scrm_{(g,\eta)}$ is the quotient of $\mu^{-1}(0)$ by the gauge group
\begin{equation}\label{eq:gauge}
\calg = \left\{ g \colon X \to S^1 \right\},
\end{equation}
which acts on $\Gamma(W^{+}) \times \cala$ as follows: 
locally, or if $X$ is simply-connected, every map $g \colon X \to S^1$ takes the form $g = e^{i f}$ for some function $f$ on $X$, and its action is given by
\begin{equation}\label{eq:gauge-action}
g \cdot (\varphi, A) = ( e^{-i f} \varphi, A + 2\, i\, \exd f ).
\end{equation}
Pick a monopole $(\varphi,A)$ and consider the differential
\begin{equation}\label{map:sw-derivative}
\exd \mu \colon T_{\Gamma(W^{+}) \times \cala}|_{(\varphi,A)} 
\to T_{\Gamma(W^{-}) \times i \,\Omega_{+}^2(X)}|_{(0,0)}\quad 
{\text{of $\mu$ at $(\varphi,A)$.}}
\end{equation}
Then, regarding the gauge action of 
$\calg$ on $(\varphi,A)$ as the map
\[
\fr{g} \colon \calg \to \Gamma(W^{+}) \times \cala, \quad 
\fr{g}(g) = g \cdot (\varphi, A),
\]
we consider the differential
\begin{equation}\label{eq:gauge-derivative}
\exd \fr{g} \colon T_{\calg}|_{f = 0} \to T_{\Gamma(W^{+}) \times \cala}|_{(\varphi,A)}\quad 
{\text{of $\fr{g}$ at $e^{i \cdot 0}$.}}
\end{equation}
Combining 
\eqref{map:sw-derivative} and \eqref{eq:gauge-derivative}, we obtain 
the complex 
\begin{equation}\label{d:main-complex}
0 \to T_{\calg}|_{f = 0} 
\to
T_{\Gamma(W^{+}) \times \cala}|_{(\varphi,A)} \to 
T_{\Gamma(W^{-}) \times i \,\Omega_{+}^2(X)}|_{(0,0)} \to 0,
\end{equation}
which, assuming the standard identifications
\[
T_{\calg}|_{f = 0} \cong i\, \Omega^{0}(X),\quad
T_{\Gamma(W^{+}) \times \cala}|_{(\varphi,A)} \cong \Gamma(W^{+}) \times i\, \Omega^1(X), \quad
T_{\Gamma(W^{-}) \times i \,\Omega_{+}^2(X)}|_{(0,0)} \cong 
\Gamma(W^{-}) \times i \,\Omega_{+}^2(X),
\]
we can write as
\begin{equation}\label{def-complex-gen}
0 \to i\, \Omega^{0}(X) \to 
\Gamma(W^{+}) \times i\, \Omega^1(X) \to 
\Gamma(W^{-}) \times i \,\Omega_{+}^2(X) \to 0.
\end{equation}
The cohomology groups of \eqref{def-complex-gen} are
\[
\calh^{0}_{(\varphi,A)} = \ker \exd \fr{g},\quad 
\calh^{1}_{(\varphi,A)} = \ker \exd \mu/\im \exd\,\fr{g},\quad 
\calh^{2}_{(\varphi,A)} = \coker \exd \mu.
\]
A monopole $(\varphi, A)$ is called reducible if 
$\varphi$ is identically zero. The stabilizer 
$(\varphi, A)$ under the action of gauge group 
is trivial unless the monopole is reducible, in which 
case the stabilizer is isomorphic to $S^1$. Thus, if 
$(\varphi,A)$ is irreducible, then $\calh^{0}_{(\varphi,A)} = 0$, and 
sequence \eqref{d:main-complex} is exact in the 
first term.
If $(\varphi,A)$ is reducible, then the Seiberg-Witten equations 
reads $F^{+}_{A} - i\eta = 0$. In that case, 
a solution only exists iff 
\begin{equation}\label{eq:reduce-1}
\langle F^{+}_{A} \rangle_{g} = i \langle \eta \rangle_{g},
\end{equation}
where the brackets in both sides denote the harmonic part of the 2-form in question. 
Since the harmonic part of $F^{+}_{A}$ depends only on the cohomology class 
of $F_{A}$, and not on the connection $A$ at hand, we may restate 
\eqref{eq:reduce-1} as
\begin{equation}\label{eq:reduce-2}
\langle \eta + 2\,\pi\,c_1(\call) \rangle_{g} = 0.
\end{equation}
Set: 
\[
\Omega_{+}^2(X)^{*} = \left\{ 
\eta \in \Omega_{+}^2(X)\,|\, \langle \eta + 2\,\pi\, c_1(\call) \rangle_{g} \neq 0
\right\}.
\]
Now, let us consider the parameterized monopole map 
\[
\mu^{\diamondsuit} \colon 
\Omega_{+}^2(X)^{*} \times \Gamma(W^{+}) \times \cala \to 
\Gamma(W^{-}) \times i \,\Omega_{+}^2(X),\qquad 
\mu^{\diamondsuit}(\eta, \varphi,A) = ( \cald^{A}\,\varphi, F^{+}_A - \sigma(\varphi) - i \eta ),  
\]
which we can regard as a family of monopoles maps parameterized 
by $\Omega_{+}^2(X)^{*}$. Define the universal moduli space as:
\[
\fr{M} = \left\{ 
(\eta,\varphi,A) \in \Omega_{+}^2(X)^{*} \times \Gamma(W^{+}) \times \cala\ |\ 
\mu^{\diamondsuit}(\eta, \varphi,A) = 0
\right\}/\sim,\quad (\eta, \varphi,A) \sim 
(\eta, g \cdot (\varphi,A))\ 
\text{for all $g \in \calg$.}
\]
Set:
\[
\pi \colon \fr{M} \to \Omega_{+}^2(X)^{*},\quad 
\pi(\eta,\varphi,A) = \eta
\]
For the proof of the following statement see \cite[Lem.\,5]{K-M}.
\begin{theorem}[Kronheimer-Mrowka, \cite{K-M}]\label{thm:km-transverse}
The map $\mu^{\diamondsuit}$ 
is transverse to the origin of $\Gamma(W^{-}) \times i \,\Omega_{+}^2(X)$, 
hence $\fr{M}$ is an infinite-dimensional manifold. The projection 
$\pi$ is a proper Fredholm map of index 
\[
d(\call) = \dfrac{1}{4}\left( c_1(\call)^2 - 2\,\chi(X) - 3\,\sigma(X) \right),
\]
and for each point $(\eta,\varphi,A) \in \fr{M}$, 
there are natural isomorphisms
\[
\ker \exd \pi|_{(\eta,\varphi,A)} = \calh^1_{(\varphi,A)},\quad
\coker \exd \pi|_{(\eta,\varphi,A)} = \calh^2_{(\varphi,A)}.
\]
\end{theorem}
\smallskip%

Following Li-Liu \cite{LL}, we now consider the monopole map 
in the more general setting of fiber bundles.
Let $B$ be a finite-dimensional manifold, 
and $\calx \to B$ be a smooth bundle over $B$ 
with fiber $X$. 
We denote the vertical tangent bundle of $\calx$ 
by $T_{\calx/B}$. 
Pick a metric on $T_{\calx/B}$, and consider the bundle 
$\fr{fr}$ of orienting orthonormal frames of $T_{\calx/B}$. 
Suppose that $T_{\calx/B}$ is given a spin$^{\cc}$ structure 
$\fr{s}$, an equivalence class of lifts of the $\SO(4)$-bundle $\fr{fr}$ 
to a $\mib{Spin}^{\cc}(4)$-bundle $\fr{s}$. Then, if we restrict $\fr{s}$ to a fiber 
$X_b$ at $b \in B$, we get a spin$^{\cc}$ structure $\fr{s}_b$ on $X_b$.
(Hereafter, given any object on the total space $\calx$, 
the object with the subscript $b$ stands for the restriction to the fiber $X_b$.) Conversely, suppose a fiber $X_b$ is given a spin$^{\cc}$ structure $\fr{s}_b$ and we want to decide whether we can extend $\fr{s}_b$ to some spin$^{\cc}$ structure $\fr{s}$ on $T_{\calx/B}$. 
The following well-known result 
(see, eg, Chapter 3 in Morgan's book \cite{Morg}) answers affirmatively this question, provided that 
$X$ is simply-connected and $B$ is a homotopy $S^2$. (This is the only case we will be considering in the sequel.)
\begin{lemma}\label{spin-extend}
Let $\calx \xrightarrow{X_b} B$ be a fiber bundle whose fiber $X_b$ 
is a smooth simply-connected manifold, and whose base $B$ 
is a homotopy $S^2$. Suppose we 
are given a spin$^{\cc}$ structure $\fr{s}_b$ on $X_b$. 
Then there exists a spin$^{\cc}$ structure $\fr{s}$ on $T_{\calx/B}$ extending 
the spin$^{\cc}$ structure $\fr{s}_b$ on $X_b$. 
\end{lemma}
\begin{proof}
Since $B$ is a homotopy $S^2$, we have a standard exact sequence:
\begin{equation}\label{exact-for-spin}
0 \to H^2(B;\zz) \to H^2(\calx;\zz) \to H^2(X_b;\zz) \to 0.
\end{equation}
Letting $\call_b$ be the determinant line bundle of $\fr{s}_b$, 
this sequence provides a lift of $c_1(\call_b) \in H^2(X_b;\zz)$ 
to a class $[S] \in H^2(\calx;\zz)$. Moreover, we can assume that 
$[S] \mod 2 = w_2(T_{\calx/B})$. Hence, there is 
a spin$^{\cc}$ structure on $T_{\calx/B}$ whose Chern class is $[S]$.
For a simply-connected manifold, the Chern class will 
distinguish any two spin$^{\cc}$ structures. Hence, the restriction of 
$\fr{s}$ to $X_b$ is the same as $\fr{s}_b$. \qed
\end{proof}
Let $\calx \xrightarrow{X_b} B$ be a fiber bundle as above and let 
$\left\{ g_b \right\}_{b \in B}$ be a family of fiberwise metrics. 
Choose a fixed spin$^{\cc}$ structure $\fr{s}$ on $T_{\calx/B}$. 
Associated to $\fr{s}$, there are spinor bundles $W^{\pm} \to B$ and 
determinant line bundle $\call = \det\,W^{+}$, which we regard as 
families of bundles
\[
W^{\pm} = \bigcup_{b \in B} W^{\pm}_{b},\quad 
\call = \bigcup_{b \in B} \call_{b}.
\]
Further, we let $\cala_{b}$ denote the space 
of $\mib{U}(1)$-connections 
on $\call_{b}$, $\calg_b$ denote the gauge groups 
acting on 
$(W^{\pm}_{b},\cala_{b})$ as stated by \eqref{eq:gauge-action}, 
and $\Omega \to B$ be the fiber bundle whose fiber $\Omega_b$ is the 
space of those $2$-forms on $X_b$ which are $g_b$-self-dual. 
Define $\fr{T} \to B$ as follows: the fiber of $\fr{T}$ 
over $b \in B$ is the space $\Gamma(W^{-}_{b}) \times i\,\Omega_{b}$. 
(Here $\fr{T}$ is for \say{target space}.) Below, we denote points 
of $\fr{T}$ by tuples $(b, \psi, i \eta)$, where $b \in B$, 
$\psi \in \Gamma(W^{-}_{b})$, $\eta \in \Omega_{b}$. Set:
\begin{equation}\label{reducible-eq}
\Omega^{*}_b = \left\{ \eta \in \Omega_{b}\ | \ 
\langle \eta + 2\,\pi\,c_1(\call_b) \rangle_{g_b} \neq 0
\right\},
\end{equation}
and let $j \colon \Omega^{*} \to B$ be the bundle whose fiber over $b \in B$ is $\Omega^{*}_b$.
\begin{lemma}\label{contract-lemma}
    The fibering $j \colon \Omega^{*} \to B$ has $(b^{+}(X) - 2)$-connected fibers. 
    If $b^{+}(X) = 1$, then those fibers consist of two connectend components, each 
    being contractible.
\end{lemma}
\begin{proof}
This is clear as 
equation \eqref{reducible-eq} cuts out an affine subspace of 
codimension $b^{+}(X)$ in $\Omega_b$. \qed
\end{proof}
\smallskip%

Define $\fr{D} \to B$ as follows: the fiber of $\fr{D}$ 
over $b \in B$ is the space 
$\Gamma(W^{+}_{b}) \times \cala_{b} \times \Omega_b^{*}$. (Here $\fr{D}$ is for \say{domain}.) 
Below, 
we denote points 
of $\fr{D}$ by tuples 
$([b,\eta], \varphi, A)$, where 
$b \in B$, $\eta \in \Omega_b^{*}$, 
$\varphi \in \Gamma(W^{+}_{b})$, $A \in \cala_{b}$. 
Consider an extended version of 
the parameterized monopole map,
\[
\mu^{\heartsuit} \colon \fr{D} \to \fr{T},\quad 
\mu^{\heartsuit}([b, \eta], \varphi, A) = 
(b, \mu_{(g_b,\eta)}(\varphi,A)),
\]
where $\mu_{(g_b,\eta)}(\varphi,A)$ is defined by \eqref{map:monopole} for the metric 
$g_b$ and perturbation $\eta$. 
In this family setting, 
the universal moduli space $\fr{M}$ is defined as follows:
\[
\fr{M} = \left\{ ([b, \eta], \varphi, A) \in \fr{D}\ |\ 
\mu^{\heartsuit}([b, \eta], \varphi, A) = (b, 0, 0)
\right\}/\sim,\quad (\eta, \varphi,A) \sim (\eta, g \cdot (\varphi,A) )
\]
for 
$(\eta, \varphi,A) \in \Omega_{b}^{*} \times \Gamma(W^{+}_b) \times \cala_{b}$ and 
$g \in \calg_b$. Set: 
\[
\pi \colon \fr{M} \to \Omega^{*},\quad 
\pi([b, \eta], \varphi, A) \to [b,\eta].
\]
As $\pi^{-1}([b,\eta])$ is precisely $\scrm_{(g_b,\eta)}$, 
we may regard 
$\fr{M}$ as a family of moduli spaces
\[
\fr{M} = \bigcup_{b \in B,\,\eta \in \Omega_{b}^{*}} \scrm_{(g_b,\eta)}.
\]
From Theorem \ref{thm:km-transverse}, we have:
\begin{theorem}[Li-Liu, \cite{LL}]\label{thm:l-l}
The projection $\pi$ is a proper Fredholm map of index 
\[
d(\call) = \dfrac{1}{4}( c_1(\call_b)^2 - 2\, \chi(X_b) - 3\, \sigma(X_b) ),
\]
and for each point $([b,\eta],\varphi,A) \in \fr{M}$, 
there are natural isomorphisms
\[
\ker \exd \pi|_{([b,\eta],\varphi,A)} = \calh^1_{(\varphi,A)},\quad
\coker \exd \pi|_{([b,\eta],\varphi,A)} = \calh^2_{(\varphi,A)}.
\]
\end{theorem}
Let $\left\{ \eta_b \right\}_{b \in B}$ be a family of fiberwise 
$g_b$-self-dual forms on $\calx$ satisfying
\begin{equation}\label{regular-families}
\langle \eta_b + 2\,\pi\,c_1(\call_b) \rangle_{g_{b}} \neq 0.
\end{equation}
Applying Sard-Smale theorem \cite{Smale}, we 
perturb $\left\{ \eta_b \right\}_{b \in B}$ so that 
it is transverse to $\pi$. Then the moduli space 
\[
\fr{M}_{(g_b,\eta_b)} = \bigcup_{b \in B} \scrm_{(g_b,\eta_b)}
\]
is either empty or a compact manifold of dimension $d(\call) + \dim B$. 
Suppose $d(\call) < 0$ and suppose $B$ is a closed  manifold of dimension 
$\left( -d(\call)\right)$. Then $\fr{M}_{(g_b,\eta_b)}$ is zero-dimensional, and 
thus consists of finitely-many points. We call
\[
\fsw_{(g_{b},\eta_{b})}(\fr{s}) = 
\# \left\{ \text{points of $\fr{M}_{(g_{b},\eta_{b})}$} \right\}\,\mod\,2
\]
the family ($\zz_2$-)Seiberg-Witten invariant of $\calx$ associated to $\fr{s}$ and $\left\{(g_{b},\eta_{b})\right\}_{b \in B}$.
\smallskip%

Let $\mathrm{R}_b$ be the space of pairs $(g_b,\eta_b)$, where 
$g_b$ is a metric on $X_b$ and $\eta_b$ is a $g_b$-self-dual $2$-form, and $\mathrm{R}_b^{*}$ be the subset of $\mathrm{R}_b$ consisting of pairs 
$(g_b,\eta_b)$ that satisfy \eqref{regular-families}. 
Note that $\mathrm{R}_b^{*}$ is homotopy equivalent to $\Omega_b^*$. Let 
$\mathrm{R}^{*} \to B$ be the fiber bundle whose fiber over $b \in B$ is 
$\mathrm{R}^{*}_{b}$, and let $\Gamma(B,\mathrm{R}^{*})$ be the space of sections 
for this bundle. If $\left\{ (g_b,\eta_b) \right\}_{b \in B}$, $\left\{ (g'_{b},\eta'_{b}) \right\}_{b \in B}$ are two families that 
are in the same connected component of $\Gamma(B,\mathrm{R}^{*})$, then Sard-Smale theorem can be applied to conclude that
\[
\fsw_{(g_{b},\eta_{b})}(\fr{s}) = \fsw_{(g'_{b},\eta'_{b})}(\fr{s}).
\]
See \cite{LL} for details. It follows from Lemma \ref{contract-lemma} that 
$\mathrm{R}^{*} \to B$ has $(b^{+}(X) - 2)$-connected fibers. Hence, $\Gamma(B,\mathrm{R}^{*})$ is connected for $b^{+}(X) > \dim B + 1$.
\begin{theorem}[Li-Liu, \cite{LL}]
If $b^{+}(X) - 1 > \dim\, B$, then $\fsw_{(g_b,\eta_b)} (\fr{s})$ 
is independent of the choice of $(g_b,\eta_b)$.
\end{theorem}
See Theorem 2.1 in \cite{LL} for 
a more general statement. We now drop the subscript $(g_b,\eta_b)$ from 
$\fsw_{(g_b,\eta_b)}(\fr{s})$ and write simply $\fsw(\fr{s})$.

\section{Seiberg-Witten for complex surfaces}\label{sec:complex}
A general reference for the Seiberg-Witten equations on K{\"a}hler surfaces is 
the book by Morgan \cite{Morg} or Nicolaescu's \cite{Nic}. Assume $X$ is a K{\"a}hler surface and $\omega$ its K{\"a}hler form with associated 
K{\"a}hler metric $g$. The complex structure 
of $X$ gives rise to a canonical 
spin$^{\cc}$ structure $\fr{s}_0$ on $X$, with determinant 
line bundle $K_{X}^{*}$, and spinor bundles
\[
W^{+} = \Lambda^{0,0} \oplus \Lambda^{0,2},\quad 
W^{-} = \Lambda^{0,1},
\]
where each term $\Lambda^{k,p}$ stands for the bundle 
of complex-valued $(k,p)$-forms on $X$. All other spin$^{\cc}$ structures 
$\fr{s}_{\varepsilon}$ on $X$ are obtained 
by taking a line bundle 
$L_{\varepsilon}$ with $c_1(L_{\varepsilon}) = \varepsilon$ and 
setting the spinor bundles to be
\begin{equation}\label{eq:spin-eps}
W^{+} = L_{\varepsilon} 
\oplus (L_{\varepsilon} \otimes \Lambda^{0,2}),\quad 
W^{-} = L_{\varepsilon} \otimes \Lambda^{0,1}.
\end{equation}
Then the determinant line bundle of $\fr{s}_{\varepsilon}$ is 
$K_{X}^{*} \otimes L_{\varepsilon}^2$. 
Reverting to the notation used in the previous section, we have
\[
\call = K_{X}^{*} \otimes L_{\varepsilon}^2, \quad 
c_1(\call) = c_1(X) + 2\,\varepsilon, \quad 
d(\call) = c_1(X) \cdot \varepsilon + \varepsilon^2.
\]
The K{\"a}hler metric $g$ induces a canonical holomorphic $\mib{U}(1)$-connection 
$A_0$ on $K^{*}_X$, and any choice of $\mib{U}(1)$-connection $B \in \calb$ on 
$L_{\varepsilon}$ combines with $A_0$ to give a connection 
$A_0 + 2\,B \in \cala$ on 
$K_{X}^{*} \otimes L_{\varepsilon}^2$. Conversely, any $\mib{U}(1)$-connection 
on $K_{X}^{*} \otimes L_{\varepsilon}^2$ is obtained that way.
\smallskip%

For a spinor $\varphi \in \Gamma(W^{+})$, we write $\varphi = (\ell,\beta)$, 
where $\ell \in \Omega^{0}(L_{\varepsilon})$ and 
$\beta \in \Omega^{0,2}(L_{\varepsilon})$. 
With this notation, the monopole map becomes 
(see, e.g., \cite[Ch.\,7]{Morg}, \cite[\S\,3.2]{Nic}):
\begin{equation}
\begin{split}
\mu \colon 
\Omega^{0}(L_{\varepsilon}) 
\oplus \Omega^{0,2}(L_{\varepsilon}) \oplus \calb 
& \to \Omega^{0,1}(L_{\varepsilon}) \oplus 
i\,\Omega^{0}(X) \omega
\oplus \Omega^{0,2}(X), \\
\mu(\ell,\beta,B) & = \left( \bar{\del}_{B} \ell + \bar{\del}^{*}_{B} \beta,\  
(F_{A}^{+})^{1,1} - \frac{i}{4}( |\ell|^2 - |\beta|^2)\omega -i \eta^{1,1} ,\  
2 F^{0,2}_{B} - \dfrac{\ell^{*} \beta}{2} - i \eta^{0,2} \right),
\end{split}
\end{equation}
where $\bar{\del}^{*} \colon 
\Omega^{0,2}(L_{\varepsilon}) \to \Omega^{0,1}(L_{\varepsilon})$ 
is the formal adjoint of 
$\bar{\del} \colon \Omega^{0,1}(L_{\varepsilon}) 
\to \Omega^{0,2}(L_{\varepsilon})$, 
$\ell^{*}$ is the image of $\ell$ under the (non-complex) isomorphism 
$L_{\varepsilon} \cong L_{\varepsilon}^{*}$ induced 
by the metric on $L_{\varepsilon}$, $\ell^{*} \beta$ is the 
image of $\ell^{*} \otimes \beta$ under the evaluation map 
$L_{\varepsilon}^{*} \otimes (L_{\varepsilon} \otimes \Lambda^{0,2}) \to 
\Lambda^{0,2}$. Here we have used the standard 
identification $i\,\Omega^2_{+}(X) \cong i\,\Omega^{0}(X) \omega \oplus \Omega^{0,2}(X)$. 
\smallskip%

Since $\cala = A_0 + 2\,\calb$, 
it is convenient to view the action of the gauge group 
$\calg$ on $W^{+}$ and $\cala$ through its action on $L_{\varepsilon}$ and $\calb$ by the formula
\[
e^{if} \cdot (\ell,B) = (e^{-if} \ell, B + i \exd f),
\]
which is chosen so to be consistent with formula \eqref{eq:gauge-action}. 
Put $\eta = -\rho^2 \omega$. Let us describe the solutions to the equation 
\begin{equation}\label{eq:kahler-sw}
\mu(\ell,\beta,B) = 0.
\end{equation}
The following two theorems are 
well-known; see, e.g., \cite[\S\,3.2]{Nic}, \cite[Ch.\,7]{Morg}.
\begin{theorem}
If $(\ell,\beta, B)$ is a solution 
to \eqref{eq:kahler-sw}, then
\[
F^{0,2}_B \equiv 0,\quad \bar{\del}_{B} \ell \equiv 0,\quad 
\bar{\del}^{*}_{B} \beta \equiv 0,
\]
and either $\ell \equiv 0$ or $\beta \equiv 0$.
\end{theorem}
The equality $F^{0,2}_B \equiv 0$ 
says that $B$ is a holomorphic connection 
on $L_{\varepsilon}$, while $\bar{\del}_{B} \ell \equiv 0$ implies 
that $\ell$ is a holomorphic section of $L_{\varepsilon}$. 
Similarly, $\bar{\del}^{*}_{B} \beta = 0$ 
says that the Hodge dual $*\beta \in 
\Gamma(L_{\varepsilon}^{*} \otimes \Lambda^{2,0})$ of $\beta$ is 
a holomorphic section of 
$L_{\varepsilon}^{*} \otimes K_{X}$.
\begin{theorem}\label{thm:witten}
Let $(\ell,\beta,B)$ be a solution 
to \eqref{eq:kahler-sw} for a large enough $\rho$. Then
\begin{enumerate}[label=\normalfont{(\alph*)}]
\item $\beta \equiv 0$.
\smallskip%

\item $\ell$ does not vanish identically.
\end{enumerate}
Conversely, if $\ell$ is 
a section of $L_{\varepsilon}$ that does not vanish identically on $X$, 
then the equation \eqref{eq:kahler-sw} admits a solution $(\ell,0,B)$. 
Furthermore, there is a one-to-one correspondence between the 
points of $\scrm_{(g,-\rho^2\,\omega)}$ and 
effective divisors in the class $\varepsilon$.
\end{theorem}
We proceed by discussing the deformation complex 
\eqref{d:main-complex} associated to the monopole map 
for the surface $X$. While it is a difficult problem to analyze the 
complex \eqref{d:main-complex} in general, there is 
an explicit description of its 
cohomology in the case of complex surfaces.
For a monopole 
$(\ell,0,B)$, the middle terms of \eqref{d:main-complex} are:
\[ 
i\,\Omega^{0}(X) 
\xrightarrow{\quad \exd \fr{g}|_{e^{i \cdot 0}}\quad }
\Omega^{0}(L_{\varepsilon}) 
\oplus \Omega^{0,2}(L_{\varepsilon}) \oplus 
i\,\Omega^1(X) 
\xrightarrow{\quad \exd \mu|_{(\ell,0,B)}\quad }
\Omega^{0,1}(L_{\varepsilon}) \oplus 
i\,\Omega^{0}(X) \omega \oplus \Omega^{0,2}(X),
\]
with the maps given by 
\begin{equation}\label{map:bloody}
\begin{split}
\exd \fr{g}|_{e^{i \cdot 0}}(i f) 
& = (-if \ell,i\,\exd f),\text{ and}\\
\exd \mu(\dot{\ell}, \dot{\beta}, \dot{B}) & = 
\left( 
\bar{\del}_B \dot{\ell} + \bar{\del}^{*}_{B} \dot{\beta} + \dot{B}^{0,1} \ell,\  
2(\exd \dot{B}^{+})^{1,1} - \frac{i}{2}\, \ell^{*} \dot{\ell} \omega,\ 
2\, \bar{\del} \dot{B}^{0,1} - \dfrac{\ell^{*} \dot{\beta}}{2}
\right),
\end{split}
\end{equation}
where
$\dot{\ell} \in \Omega^{0}(L_{\varepsilon})$, 
$\dot{\beta} \in \Omega^{0,2}(L_{\varepsilon})$, and 
$\dot{B} \in i\,\Omega^{1}(X)$. 
Here $\ell^{*} \dot{\ell}$ is 
the image of $\ell^{*} \otimes \dot{\ell} \in L_{\varepsilon}^{*} \otimes 
L_{\varepsilon}$ under the evaluation map $L_{\varepsilon}^{*} \otimes 
L_{\varepsilon} \to \overline{\cc}$.
\smallskip%

Associated to the divisor $C = \ell^{-1}(0)$, there is a natural short exact sequence:
\[
0 \to \scro_{X} \xrightarrow{\times \ell} \scro_{X}(C) \to \scro_{C}(C) \to 0,
\]
where $\scro_{X}$ is the sheaf of holomorphic 
functions of $X$, $\scro_{X}(C)$ 
is the sheaf of holomorphic sections of $L_{\varepsilon}$, and 
$\scro_{C}(C)$ is the restriction of $\scro_{X}(C)$ onto $C$. 
The map $\times \ell \colon \scro_{X} \to \scro_{X}(C)$ is the 
multiplication by $\ell$. From this short exact sequence, we have 
the associated long exact cohomology sequence:
\begin{equation}\label{d:exact-standard}
\begin{split}
0 \to H^0(X;\scro_{X}) \xrightarrow{\times \ell} H^0(X;\scro_{X}(C)) &\to 
H^0(X;\scro_{C}(C)) \to H^1(X;\scro_{X}) \xrightarrow{\times \ell} 
H^1(X;\scro_{X}(C)) \to\\
& \to H^1(X;\scro_{C}(C)) \to H^2(X;\scro_X) \to H^2(X;\scro_{X}(C)) \to 0.
\end{split}
\end{equation}
The following result is due to 
Friedman-Morgan (see \cite[Th.\,2.1]{F-M}) and Kronheimer \cite[Prop.\,4.2]{K}.
\begin{theorem}[\cite{F-M, K}]\label{thm:f-m}
If $(\ell,0,B)$ is an irreducible solution of \eqref{eq:kahler-sw}, then 
the cohomolgy groups $\calh_{(\ell,0,B)}^1$ and $\calh_{(\ell,0,B)}^2$ 
sit in an exact sequence:
\[
\begin{split}
0 \to H^0(X;\scro_{X}) \xrightarrow{\times \ell} H^0(X;\scro_{X}(C)) &\to 
\calh_{(\ell,0,B)}^1 \to H^1(X;\scro_{X}) \xrightarrow{\times \ell} 
H^1(X;\scro_{X}(C)) \to\\
& \to \calh_{(\ell,0,B)}^2 \to H^2(X;\scro_X) \to H^2(X;\scro_{X}(C)) \to 0.
\end{split}
\]
\end{theorem}
\begin{lemma}
The operator 
$T \colon \bar{\del}_B \bar{\del}^{*}_{B} + \dfrac{|\ell|^2}{4} \colon 
\Omega^{0,2}(L_{\varepsilon}) \to \Omega^{0,2}(L_{\varepsilon})$ 
is an isomorphism.
\end{lemma}
\begin{proof}
$T$ is self-adjoint; thus it suffices 
to prove that $T$ is injective. If 
$\bar{\del}_B \bar{\del}^{*}_{B} h + \dfrac{|\ell|^2}{4} h = 0$, 
then, by taking the inner product with $h$, we find:
\[
\int_{X} \langle \bar{\del}^{*}_{B} h, \bar{\del}^{*}_{B} 
h \rangle + \frac{1}{4}\int_{X} |\ell|^2 \langle h, h \rangle = 0,\quad\text{\normalfont{hence $h \equiv 0$.}}
\]
\qed
\end{proof}
Define $\delta \colon 
\Omega^{0,1}(L_{\varepsilon}) \oplus 
i\,\Omega^{0}(X) \omega \oplus 
\Omega^{0,2}(X) \to H^2(X;\scro_X)$ by
\[
\delta(\gamma, i f \omega, \nu) = \nu + \dfrac{\ell^{*}}{2} T^{-1}\left( \bar{\del}_{B} \gamma - \frac{1}{2} \nu \ell \right).
\]
The following statement is implicit in \cite{F-M}.
\begin{lemma}\label{prop:salamon}
The map $\delta$ gives a well-defined 
map $\calh^2_{(\ell,0,B)} \to H^2(X;\scro_X)$.
\end{lemma}
\begin{proof}
Suppose $(\gamma, i f \omega, \nu)$ satisfies
\begin{equation}\label{in-the-image}
\gamma = \bar{\del}_B \dot{\ell} + \bar{\del}^{*}_{B} \dot{\beta} + \dot{B}^{0,1} \ell,
\quad 
\nu = 2\, \bar{\del} \dot{B}^{0,1} - \dfrac{\ell^{*} \dot{\beta}}{2}.
\end{equation}
Then, by differentiating $\gamma$, we find:
\[
\bar{\del}_{B} \gamma =  \bar{\del}_{B} \bar{\del}^{*}_{B} \dot{\beta} +  
\bar{\del}\dot{B}^{0,1} \ell \stackrel{\text{\eqref{in-the-image}}}{=}
\bar{\del}_{B} \bar{\del}^{*}_{B} \dot{\beta} + 
\dfrac{1}{2}\nu \ell + \dfrac{|\ell|^2}{4} \dot{\beta},\ 
\text{\normalfont{which is equivalent to}}\ 
\bar{\del}_{B} \gamma - \dfrac{1}{2}\nu \ell = T(\dot{\beta}).
\]
Hence, $\delta(\gamma, i f \omega, \nu)$ is equal to $2\, \bar{\del} \dot{B}^{0,1}$, which 
is a $\bar{\del}$-exact form. \qed
\end{proof}

\section{Seiberg-Witten for Hamiltonian bundles}
The following material is well-known; 
see \cite[\S\,3.3]{Nic} for details. Let $(X,\omega)$ be a closed symplectic $4$-manifold, 
$J$ an $\omega$-compatible almost-complex structure, 
and $g(\cdot,\cdot) = \omega(\cdot,J\cdot)$ the associated Hermitian metric. 
The symplectic form $\omega$ is $g$-self-dual and of type $(1,1)$ with 
respect to $J$. The almost-complex structure $J$ gives a 
canonical spin$^{\cc}$ structure $\fr{s}_0$ with spinor bundles 
\[
W^{+} = \Lambda^{0,0} \oplus \Lambda^{0,2},\quad 
W^{-} = \Lambda^{0,1}.
\]
There exists a special connection $A_0$ on $K^{*}_{X} = \det\, W^{+}$ such that the induced 
Dirac operator is
\[
\cald^{A_0} \colon \Omega^{0,0}(X;\cc) \oplus \Omega^{0,2}(X;\cc) \to 
\Omega^{0,1}(X;\cc)\quad \cald^{A_0} = \sqrt{2}\, ( \bar{\del} \oplus \bar{\del}^{*}).
\]
See Proposition 1.4.23 in \cite{Nic} for the proof that $A_0$ exists. 
As before, we choose the spin$^{\cc}$ structure 
$\fr{s}_{\varepsilon}$ as in 
\eqref{eq:spin-eps} and parameterize all connections 
on $\call = K_{X}^{*} \otimes L_{\varepsilon}^2$ as $A = A_0 + 2\,B$, with $B$ being 
a $\mib{U}(1)$-connection on $L_{\varepsilon}$. 
The unperturbed Seiberg-Witten 
equations are:
\begin{equation*}
 \begin{cases}
   \bar{\del}_{B} \ell + \bar{\del}_{B}^{*} \beta = 0, 
   \\
   F^{0,2}_{A_0} + 2\,F^{0,2}_{B} = \dfrac{\ell^{*} \beta}{2},
   \\
   (F_{A}^{+})^{1,1} + 2 (F^{+}_{B})^{1,1} = \dfrac{i}{4}\,(|\ell|^2 - |\beta|^2) \omega,
 \end{cases}
\end{equation*}
Choosing the perturbing term as
\begin{equation}\label{eq:taubes-eta}
    i\,\eta = F^{+}_{A_0} - i\, \rho^2 \omega,
\end{equation}
we obtain what's called the $\rho$-perturbed Seiberg-Witten equations. 
The following result is due to Taubes in \cite{Taub-2}:
\begin{theorem}[Taubes, \cite{Taub-2}]\label{thm:taubes}
Under the assumption $\varepsilon \neq 0$ and $\varepsilon \cdot [\omega] \leq 0$, 
the $\rho$-perturbed Seiberg-Witten equations have no solutions for 
$\rho$ sufficiently large.
\end{theorem}
The argument can 
also read from Theorem 3.3.29 in \cite{Nic}.
\smallskip%

Let $\calx \xrightarrow{X_b} B$ be a smooth fiber bundle with fiber $X$, where 
$X$ is a simply-connected $4$-manifold and $B$ is the $2$-sphere, which 
we regard as the union $\Delta^{+} \cup \Delta^{-}$ of two unit disks. 
We denote the equator $\del \Delta^{+} = \del \Delta^{-}$ by 
$\del$. As any bundle over a disk is trivial, we can build 
$\calx$ by taking two products $\Delta^{\pm} \times X$ and gluing them 
along the boundary $\del \times X$ by a loop $g_t \in \diff_0(X)$,
\[
\calx = (\Delta^{+} \times X) \bigcup (\Delta^{+} \times X)/\sim.\quad 
(e^{it}, g_t(x))_{+} \sim (e^{it}, x)_{+}.
\]
By definition, a Hamiltonian bundle is built from a loop $g_t$ in 
$\symp(X,\omega)$. Thus, if a smooth bundle $\calx$ is Hamiltonian, 
there exists a smooth family of cohomologous symplectic forms 
$\left\{ \omega_b \right\}_{b \in B}$ on the fibers 
$\left\{ X_b \right\}_{b \in B}$ of $\calx$. 
Choosing a family $\left\{ J_b \right\}_{b \in B}$ of 
$\omega_b$-compatible almost-complex structures, we 
also obtain a family of canonical Hermitian metrics 
$\left\{ g_b \right\}_{b \in B}$. (Recall here that the space of compatible almost-complex structures is non-empty and contractible. See, e.g., \cite[Prop.\,4.1.1]{McD-Sa-2}.) 
Pick a class $\varepsilon \in H^2(X;\zz)$. Let 
$\fr{s}_{\varepsilon}$ be the spin$^{\cc}$ structure on $X_b$ 
given by \eqref{eq:spin-eps}. Using Lemma \ref{spin-extend}, 
we can choose a spin$^{\cc}$ structure on $T_{\calx/B}$ whose 
restriction to each fiber $X_b$ is $\fr{s}_{\varepsilon}$. While such 
an extension is not uniquely determined by $\fr{s}_{\varepsilon}$, we 
write it $\fr{s}_{\varepsilon}$ for short. 
Considering the family of $\rho$-perturbed Seiberg-Witten equations 
parameterized by $b \in B$, we have:
\begin{lemma}\label{prop:vanish-for-bundles}
Let $(X,\omega)$ be a closed simply-connected symplectic 
$4$-manifold, and 
let $\calx \to B$ be a Hamiltonian fiber bundle 
with fiber $X$ symplectomorphic to $(X,\omega)$, where $B$ is the $2$-sphere. 
Suppose that $\varepsilon \in H^2(X;\zz)$ satisfies
\[
\varepsilon \neq 0,\quad \varepsilon \cdot [\omega] \leq0,\quad 
c_1(X) \cdot \varepsilon + \varepsilon^2 = -2.
\]
Then $\fsw_{(g_{b},\eta_{b})}(\fr{s}_{\varepsilon}) = 0$, where 
$\eta_b$ chosen as in \eqref{eq:taubes-eta} for 
$\rho$ large enough. If $b^{+}(X) > 3$, then 
$\fsw(\fr{s}_{\varepsilon}) = 0$ for $\calx$.
\end{lemma}
\begin{proof}
Follows from Theorem \ref{thm:taubes}. \qed
\end{proof}

\section{The Kodaira-Spencer map}\label{sec:kodaira}
The following material is well-known; see, eg, work of Griffiths \cite{Griff}. 
Let $\left\{ X_t\right\}_{t \in \Delta}$, where $\Delta \subset \cc$ is 
a complex disk, be a complex-analytic 
family of compact K{\"a}hler surfaces. 
More precisely, we assume given a complex $3$-fold $\calv$, together 
with a proper, maximal rank holomorphic 
map $p \colon \calv \to \Delta$ such that $p^{-1}(t)= X_t$. 
Let $C \subset X_0$ be a smooth rational curve which is embedded 
in $\calv$ as a $(-1,-1)$-curve. 
\smallskip%

Let $\scro_{\calv}(T_{\calv})$ be the 
sheaf of holomorphic sections of $T_{\calv}$, $\scro_{X_0}(T_{\calv})$ the restriction of $\scro_{\calv}(T_{\calv})$ to $X_0$, and $\scro_{X_0}(T_{X_0})$ the sheaf of holomorphic 
sections of $T_{X_0}$. The short exact sequence  
\[
0 \to \scro_{X_0}(T_{X_0}) \to \scro_{X_0}(T_{\calv}) \to \scro \to 0,
\]
where $\scro$ is regarded as the sheaf of sections of the trivial bundle $T_{\calv}/T_{X_0}$, gives 
the cohomology long exact sequence
\begin{equation}\label{d:k-s}
\ldots \to H^0(X_0; \scro_{X_0}(T_{\calv})) \to H^0(X_0;\scro) \to H^1(X_0;\scro_{X_0}(T_{X_0})) \to \ldots
\end{equation}
Let $\left[ \dfrac{\del}{\del t} \right]$ be a generator of $H^0(X_0;\scro) \cong \cc$. 
By definition, the Kodaira-Spencer class 
\[
\rho \in H^1(X_0;\scro_{X_0}(T_{X_0}))
\]
of the family $\calv$ at $t = 0$ is the image of $\left[ \dfrac{\del}{\del t} \right]$ under the connecting homomorphism of \eqref{d:k-s}.
\begin{lemma}\label{rho-nonzero}
$\rho$ is non-trivial.
\end{lemma}
\begin{proof}
Let $\scro_{C}(T_{C})$ be the 
sheaf of holomorphic sections of $T_C$, 
$\scro_{C}(T_{X_0})$ the restriction 
of $\scro_{X_0}(T_{X_0})$ to $C$, and $\scro_{C}(T_{\calv})$ the restriction 
of $\scro_{X_0}(T_{\calv})$ to $C$. 
We have the following commutative diagram:
\begin{equation}\label{d:many-sheaves}
\begin{tikzcd}
0 \arrow{r}{} & \scro_{X_0}(T_{X_0}) \arrow{d}{} \arrow{r}{} & \scro_{X_0}(T_{\calv}) 
\arrow{d}{} \arrow{r}{} & \scro \arrow{d}{} \arrow{r}{} & 0\\
0  \arrow{r}{} & \scro_{C}(T_{X_0}) \arrow{d}{} \arrow{r}{} & \scro_{C}(T_{\calv}) 
\arrow{d}{} \arrow{r}{} & \scro \arrow{d}{} \arrow{r}{} & 0\\
0  \arrow{r}{} & \scro_{C}(N_{C|X_0}) \arrow{r}{} & \scro_{C}(N_{C|T_{\calv}}) 
 \arrow{r}{} & \scro \arrow{r}{} & 0\,,
\end{tikzcd}
\end{equation}
where $N_{C|X_0}$ is the normal bundle to $C$ in $X_0$ and 
$N_{C|\calv}$ is the normal bundle to $C$ in $\calv$. 
The cohomology diagram of \eqref{d:many-sheaves} is written:
\begin{equation}\label{d:many-cohomogy}
\begin{tikzcd}
\ldots \arrow{r}{} & 
H^0(X_0; \scro_{X_0}(T_{\calv})) \arrow{d}{} \arrow{r}{} & 
H^0(C;\scro) \arrow{d}{} \arrow{r}{} & 
H^1(X_0;\scro_{X_0}(T_{X_0})) \arrow{d}{} \arrow{r}{} & \ldots\\
\ldots \arrow{r}{} & 
H^0(C; \scro_{C}(T_{\calv})) \arrow{d}{} \arrow{r}{} & 
H^0(C;\scro) \arrow{d}{} \arrow{r}{} & 
H^1(C;\scro_{C}(T_{X_0})) \arrow{d}{} \arrow{r}{} & \ldots\\
\ldots \arrow{r}{} & 
H^0(C; \scro_{C}(N_{C|\calv}))  \arrow{r}{} & 
H^0(C;\scro) \arrow{r}{} & 
H^1(C;\scro_{C}(N_{C|X_0})) \arrow{r}{} & \ldots\,.\\
\end{tikzcd}
\end{equation}
The assumption that $C$ is of negative self-intersection implies 
that $H^0(C; \scro_{C}(N_{C|\calv})) = 0$. 
Let $\kappa \in H^1(C;\scro_{C}(N_{C|X_0}))$ be the restriction 
of $\rho \in H^1(X_0;\scro_{X_0}(T_{X_0}))$ to $H^1(C;\scro_{C}(N_{C|X_0}))$. 
Then $\kappa$ is also the image of $\left[ \dfrac{\del}{\del t} \right] \in H^0(X_0;\scro)$ 
under the connecting homomorphism 
of the last row of \eqref{d:many-cohomogy}. That homomorphism is 
injective, since $H^0(C; \scro_{C}(N_{C|\calv})) = 0$. So $\kappa$ is non-zero, 
and neither is $\rho$. \qed
\end{proof}
\medskip%

The family $\calv$ is trivial as a $C^{\infty}$-family; i.e. we can find 
a smooth fiber-preserving diffeomorphism
\begin{equation}
\begin{tikzcd}\label{d:trivial-v}
\Delta \times X_0 \arrow{d}{} \rar{}  & \calv \arrow{d}{p} \\
\Delta \arrow{r}{\id} & \Delta\,,
\end{tikzcd}
\end{equation}
and then, using this trivialization, we regard the complex 
structures on $\left\{ X_t \right\}_{t \in \Delta}$ as 
a family $\left\{ J_t \right\}_{t \in \Delta}$ 
of complex structures on $X_0$. 
Let $L \to X_0$ be the holomorphic line bundle corresponding to the divisor 
$C$, and let $B$ be any $\mib{U}(1)$-connection on $L$ which agrees 
with the holomorphic structure on $L$. To shorten 
notation, we set: 
\[
\del_t \ell = 
\dfrac{1}{2} \left( 
\exd_{B}\, \ell - i\,\exd_{B} \circ J_t \right),\quad 
\bar{\del}_t \ell = \dfrac{1}{2} \left( 
\exd_{B}\, \ell + i\,\exd_{B} \circ J_t \right),
\]
and
\[
\del = \del_0,\quad \bar{\del} = \bar{\del}_0.
\]
Let $F_B$ be the curvature of $B$, and let 
$(F^{0,2}_B)_t$ be the $(0,2)$-component of $F_B$ 
with respect to $J_t$. We have:
\[
(F^{0,2}_B)_t = 0 + \dot{F}^{0,2}_B\,t + O(t^2),
\]
where the class $[\dot{F}^{0,2}_B] \in H^2(X_0;\scro_X)$ 
does not depend on the choice of $B$ nor it depends 
on the trivialization \eqref{d:trivial-v}. 
\begin{lemma}\label{F-lemma}
If $p_g(X_0) > h^{0}(X_0; \scro_X(K_{X_0}-C))$, then 
$[\dot{F}^{0,2}_B]$ is non-trivial.
\end{lemma}
\begin{proof}
Define $\dot{J}$ by
\[
J_t = J_0 + \dot{J}\,t + O(t^2).
\]
The \say{almost-complex} condition $J_t^2 = -\id$ implies that 
$\dot{J}$ and $J_0$ anti-commute, hence $\dot{J}$ 
can be thought as an element of $\Omega^{0,1}(X_0;T_{X_0})$; 
the \say{integrability} condition $\bar{\del}_t\bar{\del}_t = 0$ implies 
(see \cite[Lem.\,(1.8)]{Griff}) that $\dot{J}$ is $\bar{\del}$-closed as 
an element of $H^1(X_0;T_X)$. 
Under the Dolbeault isomorphism, $\dot{J}$ corresponds, up to a scalar multiple, 
to the Kodaira-Spencer class $\rho$. See Lemma (1.10) in \cite{Griff} for the proof.
Restricting $\dot{J}$ onto the curve $C$ and composing it with 
the natural projection $T_{X_0}|_{C} \to N_{C|X_0}$, we get the class 
$\kappa \in H^1(C;\scro_C(N_{C|X_0}))$ defined in Lemma \ref{rho-nonzero}. 
\smallskip%

We shall obtain an explicit Dolbeault representaive of $\kappa$. 
Letting $\ell$ be a holomorphic 
section of $L$ that vanishes along $C$, we get:
\begin{equation}\label{d-B-ell}
\exd_{B}\,\ell = \del\,\ell.
\end{equation}
Let $(\exd_{B}\,\ell)|_{C}$ be the restriction of $\exd_{B}\,\ell$ 
to $C$. By \eqref{d-B-ell}, $(\exd_{B}\,\ell)|_{C}$ vanishes 
along $T_{C}$. It follows that 
$(\del\,\ell)|_{C}$ gives an isomorphism between 
$\scro_C(N_{C|X_{0}})$ and $\scro_{C}(C)$. Under this isomorphism, 
the element $\kappa$ becomes
\[
\kappa = [\del \ell \circ \dot{J}|_{C}] \in H^1(X_0;\scro_{C}(C)).
\]
Hence, $\kappa$ sits in diagram \eqref{d:exact-standard},
\begin{equation}\label{d:exact-standard-again}
\begin{split} 
\ldots \to H^1(X_0;\scro_{X_0}(C)) \to H^1(X_0;\scro_{C}(C)) \to H^2(X_0;\scro_{X_0}) \to H^2(X_0;\scro_{X_0}(C)) \to 0.
\end{split}
\end{equation}
Let us compute the image of $\kappa \in H^1(X_0;\scro_{C}(C))$ under the connecting 
homomorphism of \eqref{d:exact-standard-again}. To that end, 
we pick an extension of $\kappa$ to 
a $C^{\infty}$-section of $\Lambda^{0,1} \otimes L$. 
Specifically, we pick 
$\del \ell \circ \dot{J} \in \Omega^{0,1}(X_0;L)$. Differentiating $\del \ell \circ \dot{J}$ gives
\[
\bar{\del} \left( \del \ell \circ \dot{J} \right) = Q \ell\quad 
\text{\normalfont{for some $Q \in \Omega^{0,2}(X_0;\cc)$.}}  
\]
The element $[Q] \in H^2(X_0;\scro_X)$ 
is the desired image. 
On the other hand, using the standard identities
\begin{equation}\label{eq:t-series}
    2\,\bar{\del}_t \ell = i \left( \del \ell \circ \dot{J} \right) t + O(t^2),\quad 
    \bar{\del}_t \bar{\del}_t \ell = \dot{F}_B^{0,2} \ell\, t + O(t^2), \quad 
    \bar{\del}_t \bar{\del}_t \ell - \bar{\del} \bar{\del}_t \ell = O(t^2),
\end{equation}
we get 
\[
\bar{\del} \left( \del \ell \circ \dot{J} \right)  = 
-2\,i \dot{F}^{0,2}_B \ell,\quad\text{hence $Q = -2\,i \dot{F}^{0,2}_B$.}
\]
By the adjunction formula, 
the restriction of $K_{X_0}$ to $C$ is trivial. 
Thus, by restricting the sections of $K_{X_0}$ to $C$, we do not 
obtain more than a one-dimensional space of sections; this gives either
\[
h^{0}(X_0; \scro_X(K_{X_0}-C)) = p_g(X_0) - 1\quad\text{or}\quad 
h^{0}(X_0; \scro_X(K_{X_0}-C)) = p_g(X_0).
\]
From the assumptions made about $h^{0}(X_0; \scro_X(K_{X_0}-C))$, we get: 
\begin{equation}\label{K-C-sections}
h^{0}(X_0; \scro_X(K_{X_0}-C)) = p_g(X_0) - 1.
\end{equation}
We have
\begin{equation}\label{rr-formula}
h^{0}(\scro_{X_0}(C)) - h^{1}(\scro_{X_0}(C)) + h^{2}(\scro_{X_0}(C))) = p_g(X_0),
\end{equation}
using the Riemann-Roch formula. Applying Serre duality to \eqref{rr-formula} gives 
\[
h^{0}(\scro_{X_0}(C)) - h^{1}(\scro_{X_0}(C)) + h^{0}(X_0; \scro_{X_0}(K_{X_0}-C)) = p_g(X_0).
\]
Substituting \eqref{K-C-sections} into \eqref{rr-formula} gives
\[
h^{0}(\scro_{X_0}(C)) - h^{1}(\scro_{X_0}(C)) = 1.
\]
Since $h^{0}(\scro_{X_0}(C)) = 1$, we have $h^{1}(\scro_{X_0}(C)) = 0$. 
It follows that the connecting 
homomorphism of \eqref{d:exact-standard-again} is injective, and hence 
$[\dot{F}^{0,2}_B] \in H^2(X_0;\scro_{X_0})$ is non-trivial. \qed
\end{proof}

\section{Proof of Theorem \ref{thm:main}}
Let $\left\{ X_t \right\}_{t \in \Delta}$, where $\Delta$ is 
a complex disk, be a complex-analytic 
family of compact simply-connected surfaces. Let $C \subset X_0$ be a smooth rational curve of self-intersection number $(-2)$. We construct a smooth 
family $\left\{ \omega_t \right\}_{t \in \Delta}$ of K{\"a}hler forms on 
$\left\{ X_t \right\}_{t \in \Delta}$. To that end, 
recall that the Kodaira classification of complex surfaces asserts that 
a complex surface is K{\"a}hler iff the first Betti number 
is even. Hence, each $X_t$ is K{\"a}hler, meaning that there exists 
some K{\"a}hler form on each fiber $X_t$. The existence of 
$\left\{ \omega_t \right\}_{t \in \Delta}$ then requires 
a partition of unity argument, combined with the following 
classical result of Kodaira and Spencer. See \cite[Th.\,15]{K-S}.
\begin{theorem}[Kodaira-Spencer, \cite{K-S}]
If $X_{t_0}$ carries 
a K{\"a}hler form, then, for a sufficiently small neighbourhood 
$U$ of $t_0$ in $\Delta$, the fiber $X_t$ over 
any point $t \in U$ admits a K{\"a}hler form. Moreover, 
given any K{\"a}hler form on $X_{t_0}$, we can choose 
a K{\"a}hler form on each fiber $X_t$, which depends differentiably 
on $t$ and which coincides for $t = t_0$ with the given K{\"a}hler 
form on $X_{t_0}$.
\end{theorem}
Let $L \to X_0$ be the holomorphic line bundle corresponding to the divisor 
$C$, let $B$ be any $\mib{U}(1)$-connection on $L$ which agrees 
with the holomorphic structure on $L$. Let $F_{B} \in \Omega^{1,1}(X_0;\cc)$ be the curvature form of $B$, and let $[\dot{F}^{0,2}_B] \in H^2(X_0;\scro_X)$ be as in \S \ref{sec:kodaira}. 
Assume that 
\[
[\dot{F}^{0,2}_{B}] \neq 0.
\] 
Then there is a smaller disk $U \subset \Delta$ such that for each 
$t \in U - \left\{ 0 \right\}$, $[C] \in H^2(X_t;\zz)$ is not a $(1,1)$-class.  
By passing to the smaller disk $U$ if necessary, we assume that 
$\Delta = U$. 
\begin{lemma}\label{prop:no-curves}
Let $\left\{ X_t \right\}_{t \in \Delta}$ be the family of surfaces 
described above. Then
\begin{enumerate}[label=\normalfont{(\alph*)}]
\item If $t \in \Delta - \left\{ 0 \right\}$, then 
$X_t$ contains no effective divisors representing $\pm [C]$.
\smallskip%

\item $X_0$ contains a single effective divisor representing $[C]$ and contains 
no effective divisors representing $(-[C])$.
\end{enumerate}
\end{lemma}
\begin{proof}
For each $t \in \Delta - \left\{ 0 \right\}$, $[C] \not\in H^{1,1}(X_t;\rr)$, and 
(a) follows. To prove (b), recall that $C$ is smooth and has negative self-intersection 
number; thus there exists at most one divisor equivalent to $C$. Since
\[
\int_C \omega_0 > 0,
\]
it follows that $(-C)$ cannot be effective. \qed
\end{proof}
\smallskip%

Let $\left\{ g_t \right\}_{t \in \Delta}$ be 
the family of K{\"a}hler metrics corresponding to $\left\{ \omega_t \right\}_{t \in \Delta}$. 
For $\varepsilon \in H^2(X_0;\zz)$, let 
$\fr{s}_{\varepsilon}$ be the spin$^{\cc}$ structure on $X_0$ given by \eqref{eq:spin-eps}. 
Choose a spin$^{\cc}$ structure on 
$\left\{ X_t \right\}_{t \in \Delta}$ extending $\fr{s}_{\varepsilon}$. 
\begin{lemma}\label{lem:no-monopoles}
(The notation are as in \S\,\ref{sec:family}.) If 
$\varepsilon = -[C]$, then the parameterized moduli space 
\[
\bigcup_{t \in \Delta} \scrm_{(g_t, -\rho^2\,\omega_t)}
\]
over the disk 
\[
\Delta \to \Omega^{*},\quad t \to (t, -\rho^2\,\omega_t)
\]
is empty for $\rho$ large enough.
\end{lemma}
\begin{proof}
Follows from Theorem \ref{thm:witten} and Lemma \ref{prop:no-curves}. \qed
\end{proof}
\smallskip%

On the other hand, we have:
\begin{theorem}[Kronheimer, \cite{K}]\label{thm:transverse}
If $\varepsilon = [C]$, the parameterized moduli space
\[
\bigcup_{t \in \Delta} \scrm_{(g_t, -\rho^2\,\omega_t)},
\]
consists of a single point corresponding to the divisor $C$; the image 
of this point under $\pi$ is $(0, -\rho^2\,\omega_0)$. 
Furthermore, $\pi$ is transverse to the disk
\begin{equation}\label{my-disk}
\Delta \to \Omega^{*},\quad t \to (t, -\rho^2\,\omega_t)
\end{equation}
at $(0, -\rho^2\,\omega_0)$.
\end{theorem}
\begin{proof}
This is a special case of Proposition 3.2 of \cite{K}, 
proved in full in \S\,4 in the same paper. The only thing to prove is the transversality of $\pi$, 
as the rest follows by Lemma \ref{prop:no-curves}. Let $([0,-\rho^2\, \omega_0], \varphi, A) \in \fr{D}$ 
be a gauge representative of the only point of 
$\scrm_{(g_0, -\rho^2\,\omega_0)}$. By Theorem \ref{thm:f-m}, $\calh^1_{(\varphi,A)} = 0$; thus it suffices to show that the image of $\pi$ is not tangent to \eqref{my-disk}. Choose a map $([0,-\rho^2\, \omega_0], \varphi, A) \colon \Delta \to \fr{D}$ with $\varphi(0) = \varphi$, $A(0) = A$. We shall prove that 
the element
\begin{equation}\label{not-tangent}
\left[ \left. 
\dfrac{\exd}{\exd t}\right\vert_{t = 0}\, 
\mu^{\heartsuit}([t,g_t,-\rho^2\,\omega_t], \varphi(t), A(t))
\right]
 \in \calh^2_{(\varphi,A)}.
\end{equation}
is non-trivial. 
With the notation of \S\,\ref{sec:complex} and \S\,\ref{sec:kodaira}, we have:
\[
(\varphi,A) = (\ell,0,B),\quad (\varphi(t),A(t)) = (\ell(t),\beta(t),B(t)),
\]
\[
\mu_{(g_t, -\rho^2\, \omega_t)}(\ell(t),\beta(t),B(t)) = 
\left( \bar{\del}_t \ell(t) + \bar{\del}^{*}_t \beta(t), \ldots,\  
2\,(F^{0,2}_{B})_t - \dfrac{\ell^{*}(t) \beta(t)}{2} \right) 
\]
Putting, without loss of generality, $\ell(t) \equiv \ell$, 
$\beta(t) \equiv 0$, we get:
\[
\mu_{(g_t, -\rho^2\, \omega_t)}(\ell(t),\beta(t),B(t)) = \left( \bar{\del}_t \ell, \ldots,\  
2\,\dot{F}^{0,2}_{B} 
\right) + O(t^2).
\]
We now compute 
$\delta(\bar{\del}_t \ell, \ldots, 2\,\dot{F}^{0,2}_{B})$, 
where $\delta$ is the map defined in \S\,\ref{sec:complex}. 
Using \eqref{eq:t-series}, we get:
\[
\delta(\bar{\del}_t \ell, \ldots, 2\,\dot{F}^{0,2}_{B}) = 
2\,[\dot{F}^{0,2}_{B}] \in H^2(X_0;\scro_X).
\]
The assumption that $[\dot{F}^{0,2}_{B}] \neq 0$ implies that 
\eqref{not-tangent} is a non-zero, and the theorem follows. \qed
\end{proof}
\medskip%

Now suppose further that $C$ is a $(-1,-1)$-curve. Suppose further that $p_g(X_0) > h^{0}(X_0; \scro_X(K_{X_0}-C))$. It follows from Lemma \ref{F-lemma} that 
$[\dot{F}^{0,2}_{B}] \neq 0$, and Theorem \ref{thm:transverse} can be applied.
\smallskip%

Let $\left\{ X'_t \right\}_{t \in \Delta}$ be the complex-analytic 
family obtained from $\left\{ X_t \right\}_{t \in \Delta}$ by the 
elementary modification with center $C$. Let $C' = \rho(C)$ be the $(-1,-1)$-curve 
as in Theorem \,\ref{thm:burns}. We furnish $\left\{ X'_t \right\}_{t \in \Delta}$ with a family of K{\"a}hler forms 
$\left\{ \omega'_t \right\}_{t \in \Delta}$. 
Recall that the set of K{\"a}hler forms on a 
K{\"a}hler manifold is a convex cone. Hence, we may deform 
$\left\{ \omega'_t \right\}_{t \in \del \Delta}$ so that
\begin{equation}\label{eq:symplectic-forms-equal}
\text{$\rho_C \colon (X_t,\omega_t) \to (X'_t,\omega'_t)$ is a symplectomorphism for each $t \in \del \Delta$, }
\end{equation}
Considering 
the family of K{\"a}hler metrics 
$\left\{ g'_t \right\}_{t \in \Delta}$ associated to 
$\left\{ \omega'_t \right\}_{t \in \Delta}$, we have:
\begin{lemma}\label{prop:element-tr-applied}
Lemma \ref{prop:no-curves} holds for 
$\left\{ X'_t \right\}_{t \in \Delta}$ with $C$ replaced by $C'$ 
and $\left\{ g_t, \omega_t \right\}_{t \in \Delta}$ replaced 
by $\left\{ g'_t, \omega'_t \right\}_{t \in \Delta}$; and likewise for 
Theorem \ref{thm:transverse} and Lemma \ref{lem:no-monopoles}.
\end{lemma}
Gluing the families 
$\left\{ X_t \right\}_{t \in \Delta}$ and 
$\left\{ X'_t \right\}_{t \in \Delta}$ by the map $\rho_{C}$,
\[
\left( \Delta \times X_t \right) \cup \left( \Delta \times X'_t \right)/\sim,\quad 
(t, x) \sim (t, \rho_C(x)),\ t \in \del \Delta,\ \text{as in \eqref{eq:make-W},}
\]
we obtain a (not complex-analytic but merely differentibale) family 
$\left\{ X_s \right\}_{s \in S^2}$ of K{\"a}hler surfaces 
parameterized by the sphere $S^2 = \Delta \cup_{\id} \Delta$. 
From \eqref{eq:symplectic-forms-equal} 
we see that the families $\left\{ \omega_t \right\}_{t \in \Delta}$ and 
$\left\{ \omega'_t \right\}_{t \in \Delta}$ form 
a family of K{\"a}hler forms $\left\{ \omega_s \right\}_{s \in S^2}$.  
\smallskip%

Regarding $\left\{ X_t \right\}_{t \in \Delta}$ and 
$\left\{ X'_t \right\}_{t \in \Delta}$ as subfamilies of 
$\left\{ X_s \right\}_{s \in S^2}$, we let $X_{s_0}$ be 
$X_0$, and let $X_{s'_0}$ be $X'_0$. As $\pi_1(S^2) = 0$, 
we may canonically identify $H_2(X_{s_0};\zz)$ with 
$H_2(X_{s'_0};\zz)$. With this identification at hand, we apply 
formula \eqref{eq:reid} to get:
\[
[C] = -[C'].
\]
Let us consider the family Seiberg-Witten invariants of 
$\left\{ X_s \right\}_{s \in S^2}$. Let $\fr{s}_{[C]}$ be 
the spin$^{\cc}$ structure on $X_{s_0}$ 
given by \eqref{eq:spin-eps} with $\varepsilon = [C]$. Similarly, 
let $\fr{s}_{-[C]}$ be the spin$^{\cc}$ structure on $X_{s_0}$ with 
$\varepsilon = -[C]$. Choose spin$^{\cc}$ structures on 
$\left\{ X_s \right\}_{s \in S^2}$ extending $\fr{s}_{[C]}$ and $\fr{s}_{-[C]}$. 
We use $\fr{s}_{[C]}$, $\fr{s}_{-[C]}$ to denote 
these extensions, also. Combining Theorem \ref{thm:transverse} and 
Lemma \ref{prop:element-tr-applied}, we get: 
\[
\fsw(\fr{s}_{[C]}) = \fsw(\fr{s}_{-[C]}) = 1.
\]
If $p_g(X_0) > 1$, then $b^{+}(X) = 2\,p_g + 1 > 3$, and the 
family invariants are well defined. The family 
of surfaces $\left\{ X_s \right\}_{s \in S^2}$ cannot carry a 
family of cohomologous symplectic forms, as this would contradict 
Lemma \ref{prop:vanish-for-bundles}. This completes the proof.

\bibliographystyle{plain}
\bibliography{references}

\end{document}